\documentclass[letterpaper, 11pt]{article}  % Comment this line out if you need a4paper

\usepackage{graphicx}
\usepackage{amssymb}
\usepackage{amsmath}
\usepackage{algorithmic,algorithm}
\usepackage{subfig}
\usepackage{dsfont}
\usepackage[margin=1in]{geometry}
\def\qed{\hfill\ensuremath{\square}}

\newtheorem{theorem}{Theorem}[section]

\newtheorem{assumption}{Assumption}[section]
\newtheorem{corollary}{Corollary}[section]
\newtheorem{lemma}{Lemma}[section]
\newtheorem{proposition}{Proposition}[section]
\newtheorem{definition}{Definition}[section]

\newenvironment{proof}{\noindent{\it Proof:}\rm}{\hfill\hfill$\Box$\par\medbreak}
\newcommand{\real}{\mathbb{R}}

\newcommand{\D}{\mathcal{D}}

\newcommand{\kk}{\mathbf{k}}
\newcommand{\y}{\mathbf{y}}
\newcommand{\argmin}{\operatorname{arg min}}
\newcommand{\argmax}{\operatorname{arg max}}

\newcommand{\Prob}{\mathbb{P}}

\newcommand{\abs}[1]{|#1|}

\def\qed{\hfill\ensuremath{\square}}
\newcommand{\e}{\mathrm{e}}

\def\prob{\mathbb{P}}

\def\real{\mathbb{R}}

\def\naturals{\mathbb{N}}
\newcommand{\until}[1]{\{1,\dots, #1\}}

\newcommand{\supscr}[2]{#1^{\textup{#2}}}
\newcommand{\setdef}[2]{\{#1 \; | \; #2\}}

\newcommand{\union}{\operatorname{\cup}}

\newcommand{\subj}{\text{subj. to}}

\newcommand\oprocendsymbol{\hbox{$\square$}}
\newcommand\oprocend{\relax\ifmmode\else\unskip\hfill\fi\oprocendsymbol}

\title{Scenario Approach for Robust Blackbox Optimization\\ in the Bandit Setting}

\author{Shaunak D. Bopardikar\hspace*{1cm}\and Vaibhav Srivastava\thanks{Corresponding author: Shaunak~D.~Bopardikar. Email: \texttt{bshaunak@gmail.com}. Shaunak D. Bopardikar was supported by United Technologies Research Center. Vaibhav Srivastava is with the Department of Electrical and Computer Engineering, Michigan State University.}}

\begin{document}

\maketitle

\begin{abstract}
This paper discusses a scenario approach to robust optimization of a blackbox function in a bandit setting. We assume that the blackbox function can be modeled as a Gaussian Process (GP) for every realization of the uncertain parameter. We adopt a scenario approach in which we draw fixed independent samples of the uncertain parameter. For a given policy, i.e., a sequence of query points and uncertain parameters in the sampled set, we introduce a notion of regret defined with respect to additional draws of the uncertain parameter, termed as scenario regret under re-draw. We present a scenario-based iterative algorithm using the upper confidence bound (UCB) of the fixed independent scenarios to compute a policy for the blackbox optimization. For this algorithm, we characterize a high probability upper bound on the regret under re-draw for any finite number of iterations of the algorithm. We further characterize parameter regimes in which the regret tends to zero asymptotically with the number of iterations with high probability. Finally, we supplement our analysis with numerical results.
\end{abstract}

\section{Introduction}
Blackbox optimization refers to the problem of optimizing a function which is described only by input output values. In a bandit setting, one can only query the function at any given point in the domain and obtain the value of the function corrupted by noise. Therefore, even if one queries the function at every point in the domain, one does not obtain the optimal value definitively. In these problems, one does not have access to the form of the function and the computation of the gradient of the function at any given value can be extremely expensive or not feasible at all. This limitation rules out several popular gradient-based algorithms for optimization. 

In many control problems related to aerospace as well as large sized commercial domain applications, blackbox functions commonly arise in the form of complex simulations being put together in software or as hardware-in-the-loop simulators. There are several sources of uncertainty in terms of choices of some parameters in blackbox functions and the blackbox functions are optimized for the worst case of the uncertainty in a max-min formulation. Standard metrics to analyze the performance of bandit problems is the notion of regret, which is the average difference between the value of the function evaluated at multiple points and the unknown optimal value of the function. This paper applies a scenario-based framework to robust blackbox optimization and performs finite-time and asymptotic analysis of a novel robust notion of regret.

Deterministic approaches to robust blackbox optimization were proposed in \cite{DB-ON-KMT:10} for the case of unconstrained optimization and in \cite{DB-ON-KMT:10b} for the case of constraints.
Among stochastic approaches to blackbox optimization, a popular approach is based on the assumption that the unknown function can be represented as a GP \cite{CER-CW:06}. Typical approaches to black-box optimization have been of an iterative nature and identify an \emph{acquisition or surrogate} function to optimize at every iteration, evaluate the blackbox at the optimum of the acquisition function and finally update the acquisition function using all data up to that iteration. Acquisition functions employed successfully so far are the UCB of a GP \cite{gpucb2010}, expected value and probability of improvement~\cite{BS-KS-ZW-RPA-NDF:16}. GP-based approaches to robust blackbox optimization have also received recent attention.  \cite{JX-PIF-SS-AM-SE:12} proposes an efficient approach by defining a GP jointly over the optimization variable and the uncertain parameter and optimizes the value of information metric at every iteration. 

The bandit approach to optimization has been explored in the control theory literature~\cite{VA-PV-JW:87,RA-MVH-DT:88,RA:95,PR-VS-NEL:14h}. The bandit setting for the blackbox optimization is similar in spirit to the Bayesian approach to the so-called multi-armed bandit problem~\cite{EK-OC-AG:12,SA-NG:12,PR-VS-NEL:14a}. Although these works employ Bayesian algorithms, a frequentist notion of regret is analyzed therein.

In the literature so far, there is a clear understanding of how the regret (difference between the current function value at every iteration and the global optimum)  behaves temporally. For a blackbox function that is defined jointly over the design parameters selected by the user and uncertain parameters that are inaccessible to the user during the experiment (and are selected later by the nature), there are no theoretical guarantees (to the best of our knowledge) on how the iterative solution will perform against the actual realization of the uncertain parameter. To this end, we adopt a scenario approach for robust optimization wherein several realizations of the uncertain parameters are sampled and the function is optimized for the worst scenario~\cite{GCC-FD-RT:11}. 

Scenario approach has been used to address computationally complex robust optimization arising in control design~\cite{vidyasagar1998,TempoBaiDabbene97,calafiore2006}. This approach has found several specific applications in which the goal is to characterize the probability with which the true uncertainty will violate the constraints defined by the sample-based solution. Examples include robust model predictive control \cite{GCC-LF:13} and its applications to energy efficient buildings~\cite{XZ-GS-DS-MM:13}, compression learning problems \cite{margellos2015}, and probabilistic solutions to large matrix games \cite{BopardikarBorriHespanha13}. 

%\subsection{Contributions}
This paper discusses a scenario approach applied to robust optimization of a blackbox function over a discrete domain in a bandit setting. We assume that the blackbox function can be modeled as a GP for every realization of the uncertain parameter. The key contributions of this paper are three-fold. First, we formalize an approach based on scenario optimization, in which we draw fixed independent samples of the uncertain parameter and examine a novel notion of regret, termed as the \emph{scenario regret under re-draw}.  Second,  we present a scenario-based iterative algorithm using the UCB) of multiple scenarios to decide at what point to evaluate the blackbox function. For this algorithm, we characterize a high probability upper bound on the regret for any finite number of iterations of the algorithm. Further, we analyze the consistency of this algorithm in terms of the conditions under which the regret tends to zero asymptotically with the number of iterations with high probability. We supplement our analysis with numerical results on a simple illustrative example.

Since the analysis in our paper leverages regret bounds in GP optimization literature~\cite{gpucb2010}, our results are explicit in terms of upper bounds on regret in the problem under uncertainty as opposed to only bounds on violation probability which has been the focus on most of scenario optimization literature. For ease of exposition, we have restricted our attention in this paper to the case of the domain of the function to be a discrete set. By introducing additional assumptions on the function to be optimized, the results here can be extended to a continuous and compact domain akin to the non-robust problem~\cite{gpucb2010}. 

%\subsection{Organization of the paper}
This paper is organized as follows. Section~\ref{sec:background} provides a brief background related to the framework used in this paper. Section~\ref{sec:problem} presents the problem formulation and formally introduces the notion of regret considered in this paper.  Section~\ref{sec:decoupled} presents the UCB based algorithm and presents technical results on upper bounds on the regret under re-draw.  Section~\ref{sec:numerics} illustrates the results using a numerical example. Section~\ref{sec:conclusion} summarizes the work done with some directions identified for future research. The appendix contains the proofs of all the theoretical results.

\section{Background Results and Notation} \label{sec:background}
We begin with a summary of the Gaussian Process Upper Confidence Bound (GP-UCB) technique, one of the Bayesian approaches to blackbox optimization and then summarize a result from the scenario approach. In contrast to the gradient-based optimization approaches, in a bandit setting only a noisy value of the objective function at the query point can be obtained. Due to presence of noise, the nominal value of function cannot be known with probability one in finite time. This leads to the so-called explore-versus-exploit trade-off in sequential decision-making: to select next query point to improve belief or to select the best point per current belief.

\subsection{GP-UCB Algorithm \cite{gpucb2010}}
Consider the problem of sequentially optimizing an unknown function $F(x)$. In each round $t$, we choose a point $x_t \in X \subset \real^d$ and evaluate the function value at $x_t$. The goal is to maximize total reward, i.e., $\max \sum_{t=1}^TF(x_t)$. A reasonable performance metric is the average regret, i.e., the excess cost due to not knowing $F$'s maximizer apriori. The instantaneous regret $r_t := F(x^*) - F(x_t)$ and the average regret after $T$ rounds is $R_T:= \frac{1}{T}\sum_{t=1}r_t$, where $x^*$ is the maximizer of $F$.

The main idea behind the GP-UCB algorithm is to maintain a surrogate of $F(x)$ using a GP. This implies that given noisy samples $\y_T:=[y_1, \dots, y_T]'$ of $F$ evaluated at $A_T := \{x_1,\dots, x_T\}$, respectively, the posterior over $F$ is a Gaussian distribution mean $\mu_T(x)$ and covariance $k_T(x,x')$ given by
\begin{align} \label{eq:GP}
\mu_T(x) &= \kk_T(x)^T(K_T + \rho^2 I)^{-1} \y_T, \nonumber \\
k_T(x,x') &= k(x,x') -  \kk_T(x)^T(K_T + \rho^2 I)^{-1} \kk_T(x'), \nonumber \\
\sigma_T(x) &= \sqrt{k_T(x,x)},
\end{align}
where $k(\cdot,\cdot)$ is the kernel function, the vector\\ $\kk_T(x) := [k(x_1, x) \ldots k(x_T,x)]^T$, $K_T$ is the positive semi-definite kernel matrix $[k(x,x')]_{x,x' \in A_T}$ and $\rho$ is the standard deviation of the Gaussian measurement noise in the samples $y$. Then the GP-UCB algorithm is summarized in Algorithm~\ref{algo:GPUCB}, where $\beta_t$ is appropriately chosen. 

\begin{algorithm}[h]
\begin{algorithmic}[1]
\STATE \textbf{Input:} GP parameters: $\mu_0(\cdot), \sigma_0(\cdot), k(\cdot, \cdot), \rho$.
\FOR{$t = 1, 2, \dots $}{
\STATE Choose $x_{t} = \underset{{x\in \real^d} }{\argmax} \;  \mu_{t-1}(x) + \sqrt{\beta_t} \sigma_{t-1}(x)$.
\STATE Obtain $y_t = F(x_{t}) + n_t$.
\STATE Update the GP parameters $\mu_t(\cdot), \sigma_t(\cdot)$ using~\eqref{eq:GP}.}
\ENDFOR
\end{algorithmic}
\caption{GP-UCB for Blackbox Optimization} \label{algo:GPUCB}
\end{algorithm}

\begin{theorem}[Regret bound (Theorems 1 \& 4, \cite{gpucb2010})]\label{thm:gpucb}
Suppose that $F$ is sampled from a known GP prior with known noise variance $\rho^2$, and that the set $X$ is finite. Let $T_* \in \until{T}$, $\beta_t := 2\log(\abs{X} \pi^2 t^2/(6\epsilon))$ and $\hat{\lambda}_1, \dots, \hat{\lambda}_{|X|}$ denote the eigenvalues of the kernel $k$ evaluated over the entire set $X \times X$. Then, the regret of the GP-UCB algorithm belongs to $O\Big(\sqrt{\frac{\beta_T\gamma_T}{T}}\Big)$ with high probability, i.e.,
\[
\Prob \Big \{ R_T \leq \sqrt{\frac{8\beta_T\gamma_T}{T \log(1+\rho^{-2})}}  \quad \forall T \geq 1 \Big \} \geq 1-\epsilon,
\]
where $\gamma_T$ satisfies
\[
\gamma_T \in O(\rho^{-2} (T\sum_{j = T_*+1}^{|X|} \hat{\lambda}_j + T_*\log(T \sum_{j=1}^{|X|} \hat{\lambda}_j) ) ). \, \, \oprocend
\]
\end{theorem}

\subsection{Probabilistic Constraint Satisfaction \cite{TempoBaiDabbene97}}
Given a function $G:{X\times V} \to \real$, for any given $\hat{x} \in X$, suppose that a multisample $\{ v_1,\dots, v_N\}$ is drawn from $V^N$ according to probability $\Prob_{V^N}$. Suppose also that 
\[
M := \max_{j \in \{1,\dots, N\}} G(\hat{x}, v_j).
\]
We are interested in quantifying the \emph{violation probability}, i.e., for a new sample $v \in V$,
\[
p_v := \Prob_{v} (v \in V \, : \, G(\hat{x} , v) > M). 
\]
The result below quantifies the number of samples $N$ required so that the violation probability is below a specified threshold.

\begin{theorem}[Violation Probability Bound~\cite{TempoBaiDabbene97}]\label{thm:tempo}
For given $\zeta, \eta \in (0, 1)$, suppose that the number of samples
\[
N \geq \frac{\log\frac{1}{\zeta}}{\log\frac{1}{1-\eta}}.
\]
Then, with probability at least $1 - \zeta$, $p_v \leq \eta$. \oprocend
\end{theorem}

A useful corollary of this sample complexity bound for the purpose of this paper is the following.
\begin{corollary}[Violation probability]\label{cor:tempo}
Suppose that the number of samples
\[
N = \left \lceil \frac{1}{\eta}\log\frac{1}{\zeta} \right \rceil .
\]
Then, with probability at least $1 - \zeta$, $p_v \leq \eta$. \oprocend
\end{corollary}
The proof of this result follows from the simple fact that $\log(1/(1-\eta)) \geq \eta, \forall \eta \in [0,1)$.

\section{Problem set-up}\label{sec:problem}
This goal of this paper is to optimize a scalar function $y = F(x,\delta)$, where $x\in X$ is a decision / optimization variable and $\delta \in \Delta$ is an uncertain parameter, the output $y\in \real$ and the function $F:X\times\Delta \to \real$. The main challenges are:

\begin{enumerate}
\item The function $F$ is a black-box or a simulation and the computation of derivative of $F$ is very expensive, ruling out any gradient-based approaches.
\item There are uncertain parameters $\delta$ which might be stochastic with unknown distribution, but we can draw independent and identically distributed samples from this unknown distribution.
\end{enumerate}

This paper considers the following optimization problem 
\begin{align} \label{prob:robust}
\max_{x\in X}\min_{\delta \in \Delta} F(x,\delta),
\end{align}
where the set $\Delta$ could be a continuum. Problem~\eqref{prob:robust} can be interpreted as  $\max_{x \in X,\; \tau \in \real} \; \tau$  subject to constraints $\min_{\delta \in \Delta} F(x, \delta) \ge \tau$.  
Thus, the total number of constraints in \eqref{prob:robust} are potentially infinite, thereby making the problem computationally intractable, except for some simple cases~\cite{bertsimas2011theory}. 

We make the following assumption on the uncertainty $\delta$.
\begin{assumption}[Probabilistic uncertainty] \label{as:delta}
The uncertainty $\delta$ is a random variable with probability distribution function, $\Prob_{\delta}:\Delta \to \real_{\geq 0}$. \qed
\end{assumption}

We now introduce the following assumption on $F$.
\begin{assumption}[Robust Gaussian Process]\label{as:FGP}
For any fixed value of $\delta$, the function $F(\cdot, \delta)$ is a realization of a spatial GP with the mean function equal to zero (without loss of generality) and a kernel $k^\delta(x, x')$. In the following, we will denote such a realization by $\omega \in \Omega$. \oprocend
\end{assumption}

There are three main categories of uncertainty in the blackbox optimization. The first is the \emph{endogeneous} uncertainty that corresponds to fundamental differences in different environments of interest. For example, in the context of energy efficient buildings, endogenous uncertainty may correspond to variability due to differences in architecture of different buildings. The second is the \emph{exogenous} uncertainty that corresponds variability due to external variables such as number of occupants in the building. Together, these two uncertainties describe the underlying objective function that maps to the realization of the GP in our framework. The third is the uncertainty in accessing the realized value of the objective function and is captured by the measurement noise. In the following, we capture inaccessible exogenous uncertainty via uncertain parameter and drop the term exogenous for brevity.

Since the realization of the GP is determined by both $\omega$ and $\delta$, a natural question to ask is why should we adopt robust optimization instead of  learning the actual realization of the GP using the GP-UCB algorithm? 
Recall that the value of the function at a query point is computed using a noisy blackbox simulator that may have access to endogenous uncertainty but not to exogenous uncertainty. Therefore, we want the simulated values to be robust with respect to the actual realized value.    
Furthermore, when the policy that we design using the blackbox simulator is applied to the physical system in \emph{real time},  we need to ensure that the performance guarantees (measured via regret bounds) for the blackbox simulator also hold for the physical system. Towards this end, we introduce a novel notion of robust regret. 
 
Let $D_N := \setdef{(\omega_i, \delta_i)}{i \in \until{N}}\in (\Omega \times \Delta)^N$ be the set of $N$ independent and identically drawn samples of $(\omega,\delta)$. In scenario approach to robust optimization, the robustness is computed with respect to an additional $(N+1)$-th sample of the uncertain parameter. In the context of sequential optimization, we can draw this additional sample once at the beginning of the sequential process and compute robustness at each step with respect to this additional draw. More generally, we can sample a new $(N+1)$-th sample after every $\alpha(\naturals) \to \naturals$ steps and compute robustness with respect to this additional sample. The function $\alpha(\cdot)$ is termed as the \emph{frequency of re-draw} and we make the following assumption.

\begin{assumption}\label{as:alpha}
The function $\alpha(\cdot)$ satisfies $1\leq \alpha(t) \leq t$, for each $t \in \naturals$.
\end{assumption}
 
The lower bound in Assumption~\ref{as:alpha} arises in the case when we sample a single additional scenario once at the beginning of the sequential process. The upper bound arises in the case when we sample a new $(N+1)$-th scenario at each iteration of the sequential process. This notion of robustness corresponds to probability that the max-min solution obtained using some policy with $N$ scenarios will remain the same when an additional (fresh) scenario is presented after every specific number of iterations. We refer to this notion of robustness as \emph{robustness under re-draw}.  

Assumption~\ref{as:alpha} is similar in spirit to the assumption in the multi-armed bandit problem literature on the number of break-points in an abruptly changing environment~\cite{AG-EM:08,VS-PR-NEL:14a,LW-VS:17m}. Loosely speaking, the frequency of re-draw is a time-scale separation between the sampling rate and the rate of change in the environment. Intuitively, a frequency of re-draw that is a sub-linear function of $T$ allows for infinitely many samples between two redraws in the limit $T \to + \infty$ and in this regime, the underlying functions can be learnt asymptotically.

In the following, we denote the $(N+1)$-th scenario at time $t \in \naturals$ by $d_{N+1}^t =(\omega_{N+1}^t, \delta_{N+1}^t)$. The scenario version of the problem~\eqref{prob:robust} given by
\begin{equation}\label{eq:scenario-robust}
J(D_N) : =\max_{x \in X} \min_{d \in D_N} F(x, d).
\end{equation}
Let $(x^*{[N]}, d^*[N])$ be the unique solution to~\eqref{eq:scenario-robust}, where $d^*[N]:=(\omega_{i^*[N]}, \delta_{i^*[N]})$. We define the following notion of regret corresponding to robustness under re-draw.

\begin{definition}[Scenario regret under re-draw]\label{def:regret}
 Given an algorithm generating a sequence of decisions $(x_t, i_t)$, $t \in \until{T}$ and $i_t \in \until{N}$, the \emph{scenario regret under re-draw} of the algorithm is given by
\begin{equation}\label{eq:regret-policy}
\supscr{R}{re-draw}_T(D_N \union \{d_{N+1}^1, \ldots, d_{N+1}^T\}) := \frac{1}{T} \sum_{t=1}^T ( J(D_N \union d_{N+1}^t ) - F(x_t, d_{i_t})),
\end{equation}
where, given a frequency of re-draw $\alpha$, the sequence of samples $\{d_{N+1}^1, \ldots, d_{N+1}^T\}$ satisfies
\begin{equation}\label{eq:dalpha}
d_{N+1}^{\alpha(t)} = d_{N+1}^{\alpha(t)+1} = \ldots = d_{N+1}^{\alpha(t+1)-1},
\end{equation}
for every $t\in \{1,\dots, T\}$. \oprocend
\end{definition}

Note that in $\supscr{R}{re-draw}_T$, the sequence of decisions is computed using only $N$ scenarios and the regret is computed after incorporating the $(N+1)$-th scenario. Recall that these scenarios are unknown realizations of GP and need to be learned in a bandit setting. There are three sources of randomness in the above formulations of the robustness/regret: (i) the random draw of $N$ scenarios; (ii) the random draw of $(N+1)$-th scenario(s); and (iii) the noise in function evaluation at query point.

We seek to design an algorithm and compute an associated upper bound on the above notion of regret. We will focus on finite time as well as asymptotic analysis of the regret associated with the algorithm. Although we assume that the set $X$ is finite and discrete, the asymptotic case remains of interest due to aforementioned sources of noise. We are interested in consistent algorithms that are defined as follows.
\begin{definition}[Consistent Algorithm]\label{def:consistent}
An algorithm is said to be \emph{consistent} if it generates a sequence of decisions $x_t$ and possibly $\delta_{i_t}$, such that the regret under re-draw asymptotically tends to zero as $T \to +\infty$. \oprocend
\end{definition}

\section{Scenario Regret under Re-draw}\label{sec:decoupled}
In this section, we will first establish robustness properties for the scenario regret under re-draw. Then, we will present an algorithm and establish associated guarantees on the scenario regret under re-draw.

\subsection{Robustness of scenario regret under re-draw}

\begin{proposition}[Probabilistic robustness guarantees]\label{prop:robust-regret}
Given the parameters $\eta, \zeta \in (0,1)$, the scenario regret under re-draw defined in~\eqref{eq:regret-policy} with $N = \lceil \alpha(T)/\eta\log(1/\zeta) \rceil$, satisfies
\[
\prob\big(d_{N+1}^1,\dots,d_{N+1}^{\alpha(T)} \in (\Omega \times \Delta)^{\alpha(T)} \, |  \supscr{R}{re-draw}_T(D_N \union \{d_{N+1}^1, \ldots, d_{N+1}^T\}) = \supscr{R}{re-draw}_T (D_N)\big) 
> 1-\eta,
\]
with probability at least $1- \zeta$.  \oprocend
\end{proposition}

In Proposition~\ref{prop:robust-regret}, the outer probability is with respect to $N$-sample, and the inner probability is defined with respect to a new $(N+1)$-th sample. Proposition~\ref{prop:robust-regret} implies that if we have a large number of samples of uncertain parameter $d = (\omega, \delta)$, then with high probability the average regret of a sequence of choices with respect to sampled set $D_N$ will remain the same with the addition of a sequence of new $N+1$-th samples $d_{N+1}^t$.

\subsection{A UCB based Algorithm}

In view of the robustness guarantees from Proposition~\ref{prop:robust-regret}, we now focus on solving problem~\eqref{eq:scenario-robust}. Since $F(x, d)$ is a black-box function, it is tempting to treat $x \mapsto \min_{d \in D_N} F(x, d)$ as a black-box function and apply Algorithm~\ref{algo:GPUCB} to solve~\eqref{eq:scenario-robust}. However, black-box computation of $F(x, d_i)$ for each $i \in \until{N}$ is a realization of some Gaussian random variable. The minimum of these Gaussian random variables is no longer a Gaussian random variable, and we cannot apply Algorithm~\ref{algo:GPUCB} to $\min_{d \in D_N} F(x, d)$. 

Our  approach is described in Algorithm~\ref{algo:scenarioUCB}, in which we maintain a  GP for each $i$, and only update the GP associated with (seemingly) the worst scenario ($i_t$) at each iteration. At every iteration $t$, $n_t \sim \mathcal{N}(0,\rho^2)$ are independently generated. The key idea behind Algorithm~\ref{algo:scenarioUCB} is that \emph{only one} GP realization is updated at each iteration making the implementation scalable from a computational viewpoint. 

\begin{algorithm}[h]
\begin{algorithmic}[1]
\STATE \textbf{Input:} A positive integer $N$, $\Prob_d(\cdot)$ and the blackbox function $F(\cdot, \cdot)$
\STATE Draw a multi-sample $d_1, \dots, d_N$ using $\Prob_d(\cdot)$.
\STATE Choose initial values of GP parameters $\{(\mu^{1}_0(\cdot), \sigma^{1}_0(\cdot)), \dots, (\mu^{N}_0(\cdot), \sigma^{N}_0(\cdot))\}$ and $k^{\delta_1}(\cdot, \cdot), \dots, k^{\delta_N}(\cdot, \cdot)$.
\STATE $\D_0^i = \emptyset, \forall i \in \{1,\dots, N\}$.
\FOR{$t = 1, 2, \dots $}{
\STATE Set $\displaystyle x_{t} := \underset{x\in X}{\argmax} \min_{i\in\{1,\dots, N\}} \mu^i_{t-1}(x) + \sqrt{\beta_t} \sigma^i_{t-1}(x)$
\STATE Set $\displaystyle i_t \in \underset{{i\in \{1,\dots, N\}}}{\argmin} \mu^i_{t-1}(x_t) + \sqrt{\beta_t} \sigma^i_{t-1}(x_t)$
\STATE Obtain $y_t = F(x_{t},d_{i_t}) + n_t$ 
\STATE $\D_t := \D_{t-1} \cup \{x_t, y_t, i_t\}$ 
\STATE Update GP parameters $(\mu^{i_t}_t, \sigma^{i_t}_t)$ using~\eqref{eq:GP}.
}
\ENDFOR
\end{algorithmic}
\caption{Scenario Upper Confidence Bound} \label{algo:scenarioUCB}
\end{algorithm}

We now present bounds on the regret under re-draw~\eqref{eq:regret-policy} under Algorithm~\ref{algo:scenarioUCB}. 
\begin{theorem}[Regret under re-draw] \label{thm:weak-regret}
For each $i \in \{1,\dots, N\}$, let $\hat{\lambda}_j^i$'s denote eigenvalues of the kernel matrix $k^{\delta_i}$ evaluated over the set $X \times X$. For Algorithm~\ref{algo:scenarioUCB} applied to problem~\eqref{eq:scenario-robust} with $\beta_t := 2\log(\abs{X}\pi^2t^2/(3\epsilon))$ and regret under re-draw~\eqref{eq:regret-policy}, the following statement holds.
\begin{enumerate}
\item For any $T \geq 1$, 
\begin{equation*}
\prob \Big(\supscr{R_T}{re-draw} (D_N) \le  \sqrt{ \frac{8  \beta_T \gamma_T}{T \log(1+ 1/\rho^2)}} \Big)   \ge 1 -\epsilon ,
\end{equation*}
where, $\gamma_T = \sum_{i=1}^N \gamma^i_T$ and \\ $\gamma_T^i \in O ( \rho^{-2}(T \sum_{j=2}^{|X|} \hat{\lambda}_j^i + \log \big(T \sum_{j=1}^{|X|} \hat{\lambda}_j^i \big) ) )$.
\item For all $T > |X|$, the claim in 1) holds with the choice of $\gamma_T^i \in O ( |X| \log \big(T \sum_{j=1}^{|X|} \hat{\lambda}_j^i \big)$.
\end{enumerate} \oprocend
\end{theorem}

We now summarize Proposition~\ref{prob:robust} and Theorem~\ref{thm:weak-regret}.
\begin{corollary}\label{cor:scenario-regret-ucb}
For the policy obtained from Algorithm~\ref{algo:scenarioUCB} solving problem~\eqref{eq:scenario-robust} with $N = \lceil \alpha(T)/\eta\log(1/\zeta) \rceil$ scenarios, with probability at least $1- \zeta$,
\begin{multline*}
\prob \Big( \Big\{d_{N+1}^1,\dots,d_{N+1}^{\alpha(T)} \in \Omega^{\alpha(T)} \times \Delta^{\alpha(T)} \; \Big| \; \\
\prob\Big(\supscr{R}{re-draw}_T(D_N \union \{d_{N+1}^1, \ldots, d_{N+1}^T\}) \leq \sqrt{ \frac{8  \beta_T \gamma_T}{T\log(1+1/\rho^2)}} 
 \ge 1-\epsilon \Big\}\Big)  \ge 1-\eta. \, \,  \oprocend
\end{multline*} 
\end{corollary}

\medskip

Note the three levels of probability in Corollary~\ref{cor:scenario-regret-ucb}. The outermost probability of $(1-\zeta)$ is with respect to the $N$ samples drawn in the scenario approach, the middle probability is the probability that measures robustness with respect to the additional draws $(d_{N+1}^1, \dots, d_{N+1}^{\alpha(T)})$, and the innermost probability is with respect to the temporal realizations of the $N$ measurement noise sequences. We conclude this section with the following result.

\medskip

\begin{proposition}[Consistency]\label{prop:efficiency}
With high probability, the regret scales as 
\[
O\Big (\sqrt{\frac{\alpha(T)|X|}{\eta T}\log\frac{1}{\zeta}\log\frac{|X|T^2}{\epsilon}\log(|X|T)} \Big ) 
\]
In other words, a sufficient condition for consistency of Algorithm~\ref{algo:scenarioUCB} in the sense of Definition~\ref{def:consistent}, is that $\alpha(T)$ must scale strictly sub-linearly with $T$. 
  \oprocend
\end{proposition}
The proof of this claim follows from Corollary~\ref{cor:scenario-regret-ucb} which implies that the regret under re-draw scales as $O(\sqrt{N|X|\log(|X|T^2/\epsilon)\log(|X|T)/T})$ in the general case of any arbitrary $T \geq 1$. Additionally, consistency may not be guaranteed if one draws a new $(N+1)$-th scenario at every iteration to test for robustness, which corresponds to $\alpha(T) = \xi T$, for some $\xi \in (0,1]$. 

\section{Numerical Results}\label{sec:numerics}

In this section, we provide some numerical results for Algorithm \ref{algo:scenarioUCB}. In these numerical experiments, we set $X := \{0,\, 0.01,\, 0.02,\, \dots \, 1\}$ and we choose a squared exponential kernel,
\[
k_\delta (x_1, x_2):= \exp\left (-\frac{(x_1 - x_2)^2}{(0.05+0.01\delta)^2} \right),
\]
where $\delta$ is distributed uniformly randomly in the interval $[0,1]$. The function $F$ is generated as per Assumption~\ref{as:FGP}. With the choice of $N = 20$ scenarios, Figure~\ref{fig:UCB} summarizes the actual regret of Algorithm~\ref{algo:scenarioUCB} with the choice of three different functions $\alpha$: 1) $\alpha(t) = t^{0.1}$, 2) $\alpha(t) = t^{0.4}$, and 3) $\alpha(t) = t$. 

\begin{figure}[!h]
\centering
\includegraphics[width=0.8\columnwidth]{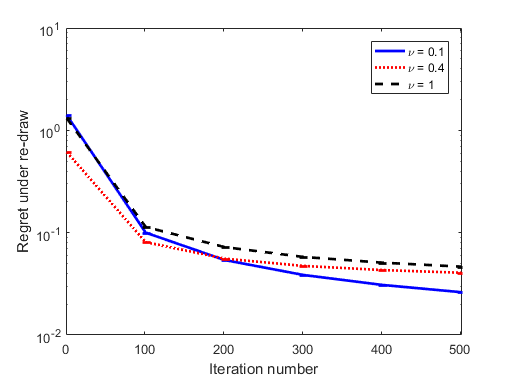}
\caption{Numerical evolution of the regret under re-draw using the UCB approach (Algorithm~\ref{algo:scenarioUCB}) for the choice of three functions $\alpha$: 1) $\alpha(t) = t^{0.1}$, 2) $\alpha(t) = t^{0.4}$, and 3) $\alpha(t) = t$.}
\label{fig:UCB}
\end{figure}

We first observe that after a few iterations the regret $\supscr{R}{re-draw}_T$ is already below $0.5$ and continues to decrease as is expected. Further, after an initial transient of about 200 iterations, the asymptotic regret suggests a monotonic trend with the value of the parameter $\nu$ in $\alpha(t) = t^\nu$. Furthermore, although Proposition~\ref{prop:efficiency} provides only a sufficient condition for $\supscr{R}{re-draw}_T$ tending to zero asymptotically for the choice of $\nu < 1$, we observe that in practice, $\supscr{R}{re-draw}_T$ tends to zero even with the choice of $\nu = 1$, with far fewer number of scenarios $N$ than in the sufficient condition. These results suggest that there is room for improved sample complexity analysis of this method.

\section{Conclusion and Future Directions}\label{sec:conclusion}
We applied the scenario approach to robust optimization of a blackbox function in a bandit setting. This paper assumed that the blackbox function can be modeled as a GP for every realization of the uncertain parameter. We developed a scenario approach in which we draw fixed independent samples of the uncertain parameter and introduced a notion of regret under re-draw. We formalized a variant of a scenario-based iterative algorithm using the UCB of multiple scenarios to decide at what point to evaluate the blackbox function. For this algorithm, we characterized a high probability upper bound on the regret under re-draw for any finite number of iterations of the algorithm and further characterized conditions under which the regret tends to zero asymptotically with the number of iterations with high probability. Finally we supplemented our analysis with numerical results on a simple example.

Future directions include improved and tighter bounds on the high probability guarantees, and improved techniques based on sequential sampling  of the scenarios.

\section*{Acknowledgments}
The authors would like to thank Dr.~Kunal Srivastava from United Technologies Research Center for several insightful conversations on optimization of blackbox functions.

%{\footnotesize
%\bibliographystyle{IEEEtran}
%\bibliography{references}}


\section*{Appendix: Mathematical Proofs}
We now present the mathematical proofs of the results.

\subsection{Proof of Proposition~\ref{prop:robust-regret}}

Problem~\eqref{eq:scenario-robust} is equivalent to
\begin{equation}\label{eq:scenario-subj-to}
\begin{split}
\underset{{x \in X, \tau \in \real}}{\max} & \quad \tau \\
\subj & \quad F(x, d_i) \ge \tau, \; \text{for all } i \in \until{N}.
\end{split}
\end{equation}
Let the optimal value of $\tau$ be $\tau^*{[N]}$. Then, $\tau^*{[N]} = F(x^*{[N]}, d_{i^*{[N]}})$. Further, let the solution of~\eqref{eq:scenario-subj-to} after addition of additional constraint associated with $d_{N+1}^t$ be $(\tau_t^*{[N+1]}, x_t^*{[N+1]})$, and $i_t^*{[N+1]}$ be an active constraint such that $\tau_t^*{[N+1]}= F(x_t^*{[N+1]}, d_{i_t^*[N+1]})$. Since adding a constraint cannot increase the objective function in~\eqref{eq:scenario-subj-to}, 
\begin{equation}\label{eq:opt-upper}
\tau_t^*{[N+1]} \le \tau^*{[N]}. 
\end{equation}

Since $\min_{d \in D_N} F(x^*{[N]}, d) = \tau^*{[N]}$, with the choice of $N = \lceil \alpha(T)/\eta\log(1/\zeta)\rceil$, it follows from Corollary~\ref{cor:tempo} that, with probability\footnote{This probability is with respect to the multi-sample $d_1,\dots, d_N$ that defines $\tau^*{[N]}$.} at least $1 - \zeta$,  
\[
\prob(d_{N+1}^t \in \Omega \times \Delta \, | \, F(x^*{[N]}, d_{N+1}^t) \ge \tau^*{[N]})
 \ge 1 - \frac{\eta}{\alpha(T)}.
\]
Furthermore, if $F(x^*{[N]}, d_{N+1}^t) \ge \tau^*{[N]}$, then 
\begin{equation}\label{eq:opt-lower}
\tau_t^*{[N+1]}= F(x_t^*{[N+1]}, d_{i^*[N+1]}) \ge \tau^*{[N]}.
\end{equation}
From~\eqref{eq:opt-upper}~and~\eqref{eq:opt-lower}, it follows that 
\begin{equation}\label{eq:probt}
\prob \big(d_{N+1}^t \in \Omega \times \Delta \, | \,   F(x^*{[N]}, d_{i^*[N]}) \\
 = F(x_t^*{[N+1]}, d_{i_t^*[N+1]}) \big) \ge 1 - \frac{\eta}{\alpha(T)},
\end{equation}
with probability at least $1 - \zeta$. Now if $F(x^*{[N]}, d_{i^*[N]})
 = F(x_t^*{[N+1]}, d_{i_t^*[N+1]})$ holds for every $t \in \{1,\dots,  T\}$, then Definition~\ref{def:regret} implies that 
\[
 \supscr{R}{re-draw}_T(D_N \union \{d_{N+1}^1, \ldots, d_{N+1}^T\}) = \supscr{R}{re-draw}_T (D_N).
\]
Further, using \eqref{eq:dalpha}, we conclude that
\begin{align*}
\prob\big(&d_{N+1}^1,\dots,d_{N+1}^{\alpha(T)} \in (\Omega \times \Delta)^{\alpha(T)} \, |  \supscr{R}{re-draw}_T(D_N \union \{d_{N+1}^1, \ldots, d_{N+1}^T\}) \neq \supscr{R}{re-draw}_T (D_N)\big) \\ &= \prob\Big( \bigcup_{k = 1}^{\alpha(T)} F(x^*{[N]}, d_{i^*[N]}) 
 \neq F(x_k^*{[N+1]}, d_{i_k^*[N+1]}) \Big)  \\
&\leq \sum_{k = 1}^{\alpha(T)} \prob(F(x^*{[N]}, d_{i^*[N]}) \neq F(x_k^*{[N+1]}, d_{i_k^*[N+1]})) \\
&\leq \alpha(T)\frac{\eta}{\alpha(T)} = \eta,
\end{align*}
with probability  at least $1 - \zeta$, where the final inequality is due to \eqref{eq:probt}. The claim follows from Definition~\ref{def:regret}.

\subsection{Proof of Theorem~\ref{thm:weak-regret}}

We now analyze Algorithm~\ref{algo:scenarioUCB}. At the $t$-th iteration, define the instantaneous regret $r$ as
\[
r(x_t, d_{i_t}) := \max_{x\in X} \min_{d \in D_N } F(x, d)- F(x_t, d_{i_t}). 
\]
We begin with establishing high probability bounds on  $r(x_t, d_{i_t})$.
\begin{lemma}[Instantaneous regret bound]\label{lem:regretUCB}
Suppose that the set $X$ is discrete and finite. Then, for any $t \geq 1$, 
\[
\Prob(r(x_t, d_{i_t}) \leq 2\sqrt{\beta_t}\sigma^{i_t}_{t-1}(x_t) | \D_{t-1}) \geq 1-2\abs{X}\e^{-\beta_t/2},
\]
for each $i \in \until{N}$, where the underlying random variable corresponds to the posterior distribution of the GP  $F(x, d_{i_t})$  conditioned on $\D_{t-1}$ defined in Algorithm~\ref{algo:scenarioUCB}. 
\end{lemma}

\begin{proof}
Assumption~\ref{as:FGP} implies that the posterior of $F(x, d_i)$ conditioned on $D_{t-1}$ is Gaussian with mean $\mu_{t-1}^{i}(x)$ and standard deviation $\sigma_{t-1}^{i}(x)$. For any $t$ and for any $i$:
\begin{align*}
F(x, d_{i}) | \D_{t-1} \sim \mathcal{N}(\mu_{t-1}^{i}(x), (\sigma_{t-1}^{i}(x))^2). 
\end{align*}
From the concentration of measures inequality for Gaussian random variables, for any $i$, we have
\[
\Prob(\abs{F(x, d_{i}) - \mu_{t-1}^{i}(x)} >  \sqrt{\beta_t}\sigma_{t-1}^{i}(x) | \D_{t-1}) \leq 2\e^{-\beta_t/2}.
\]
Using the union bound, we have for any $i$ and for all $x \in X$,
\begin{equation}\label{eq:beta}
\Prob(\abs{F(x, d_{i}) - \mu_{t-1}^{i}(x)} > \sqrt{\beta_t}\sigma_{t-1}^{i}(x) | \D_{t-1}) \leq 2\abs{X}\e^{-\beta_t/2}.
\end{equation}
Now consider
\begin{align*}
r(x_t, \delta_{i_t}) 
&\leq F(x^*[N],d_{i_t}) - F(x_t, d_{i_t}),
\end{align*}
where $x^*[N]$ is a solution of the sampled problem~\eqref{eq:scenario-subj-to}. Therefore, with probability at least $1 - 2\abs{X}\e^{-\beta_t/2}$, we have
\begin{align*}
r(x_t, d_{i_t})  &\leq \mu_{t-1}^{i_t}(x^*[N]) + \sqrt{\beta_t}\sigma_{t-1}^{i_t}(x^*[N])  -  \mu_{t-1}^{i_t}(x_t) + \sqrt{\beta_t}\sigma_{t-1}^{i_t}(x_t) \\ &\leq  \mu_{t-1}^{i_t}(x_t) + \sqrt{\beta_t}\sigma_{t-1}^{i_t}(x_t)  -  \mu_{t-1}^{i_t}(x_t) + \sqrt{\beta_t}\sigma_{t-1}^{i_t}(x_t) \\ 
&= 2\sqrt{\beta_t}\sigma_{t-1}^{i_t}(x_t),
\end{align*}
where the second inequality follows from the definition of $x_t$ in step 6 of Algorithm~\ref{algo:scenarioUCB}. This completes the proof.
\end{proof}

We now present the proof of Theorem~\ref{thm:weak-regret}.

%\begin{proof}[Proof of Theorem~\ref{thm:weak-regret}]
With $\beta_t= 2 \log(\abs{X}\pi^2t^2/(3\epsilon))$, it follows from Lemma~\ref{lem:regretUCB} that 
\[
\Prob(r(x_t, d_{i_t}) > 2\sqrt{\beta_t}\sigma^{i_t}_{t-1}(x_t)) \le 6 \epsilon/ \pi^2 t^2. 
\]
By applying union bound,
\[
\Prob(r(x_t, d_{i_t}) > 2\sqrt{\beta_t}\sigma^{i_t}_{t-1}(x_t), \; \text{for some } t \in \naturals) \le \epsilon. 
\]
Therefore, following steps similar to those in~\cite{gpucb2010}, with probability at least $1-\epsilon$,  
\begin{align*}
\sum_{t=1}^T  r(x_t, d_{i_t})^2 \le \sum_{t=1}^T 4 \beta_t (\sigma^{i_t}_{t-1}(x_t))^2 
&\le 4 \beta_T \sum_{t=1}^T \sum_{i=1}^N (\sigma^{i}_{t-1}(x_t))^2 \boldsymbol 1(i_t =i) \\
 &\le  \frac{4 \beta_T}{\log(1+1/\rho^2)} \sum_{i=1}^N \sum_{t=1}^T \log \Big(1 +\frac{(\sigma^i_{t-1}(x_t))^2}{\rho^2} \Big),
\end{align*}
where the last inequality follows since $s^2 \le \log(1+ s^2)/(\rho^2 \log(1+1/\rho^2))$ for each $s \in [0, 1/\rho^2]$~\cite{gpucb2010}. 

Let $\gamma^i_T = \frac{1}{2}\sum_{t=1}^T  \log \big(1 +\frac{(\sigma^i_{t-1}(x_t))^2}{\rho^2} \big)$. It follows from Theorem~\ref{thm:gpucb} (with the choice of  $T_* = 1$ therein) that
\[
\gamma_T^i \in O \Big ( \rho^{-2} \big (T \sum_{j=2}^{|X|} \hat{\lambda}_j^i + \log (T \sum_{j=1}^{|X|} \hat{\lambda}_j^i ) \big ) \Big ).
\]
From Cauchy-Schwartz inequality, it follows that
\begin{align*}
\supscr{R_T}{re-draw} (D_N) &\le \frac{1}{T}\sqrt{T \sum_{t=1}^T  r(x_t, d_{i_t})^2} \le \sqrt{\frac{8 \beta_T \gamma_T}{ T \log(1+1/\rho^2)}},
\end{align*}
which establishes the first claim of this theorem. To establish the second claim, observe that in the regime of $T \geq |X|$, with the choice of $T_* = |X|$ in Theorem~\ref{thm:gpucb}, we have
\[
\gamma^i_T  \in O\Big(\frac{|X|}{\rho^2} \log  \Big(T \sum_{j=1}^{|X|} \hat{\lambda}_j^i \Big) \Big),
\]
which leads to the second claim. 
%\end{proof}

\end{document}